\documentclass[11pt]{article}
\usepackage{amsmath,amsthm,amsfonts,amssymb,bm,wasysym}
\usepackage{subfigure}
\usepackage{epsfig}
\usepackage[usenames]{color}
\usepackage{verbatim}

\parskip \medskipamount
\newtheorem{theorem}{Theorem}[section]
\newtheorem{definition}[theorem]{Definition}

\numberwithin{equation}{section}
\newtheorem{lemma}[theorem]{Lemma}

\newtheorem{remark}[theorem]{Remark}

\numberwithin{equation}{section}

\def\N{\mathbb{N}}
\def\Z{\mathbb{Z}}

\def\B{\mathcal{B}}

\renewcommand{\phi}{\varphi}
\renewcommand{\epsilon}{\varepsilon}

\newcommand{\1}{{\text{\Large $\mathfrak 1$}}}

\renewcommand{\emptyset}{\varnothing}

\newcommand{\E}[1]{\mathbb{E}\!\left[#1\right]}
\newcommand{\estart}[2]{\mathbb{E}_{#2}\!\left[#1\right]}
\newcommand{\prstart}[2]{\mathbb{P}_{#2}\!\left(#1\right)}

\newcommand{\econd}[2]{\mathbb{E}\!\left[#1\;\middle\vert\;#2\right]}

\newcommand{\id}[1]{\text{indeg}(#1)}
\newcommand{\outd}[1]{\text{outdeg}(#1)}

\newcommand{\ida}[2]{\text{indeg}_{#2}(#1)}
\newcommand{\outda}[2]{\text{outdeg}_{#2}(#1)}

\newcommand\be{\begin{equation}}
\newcommand\ee{\end{equation}}

\begin{document}
\title{\bf A permuted random walk exits faster}

\author{
Richard Pymar\thanks{University College London, London, UK; r.pymar@ucl.ac.uk} \and Perla Sousi\thanks{University of Cambridge, Cambridge, UK;   p.sousi@statslab.cam.ac.uk}
}
\maketitle
\begin{abstract}
Let $\sigma$ be a permutation of $\{0,\ldots,n\}$. We consider the Markov chain $X$ which jumps from $k\neq 0,n$ to $\sigma(k+1)$ or $\sigma(k-1)$, equally likely. When $X$ is at 0 it jumps to either $\sigma(0)$ or $\sigma(1)$ equally likely, and when $X$ is at $n$ it jumps to either $\sigma(n)$ or $\sigma(n-1)$, equally likely. We show that the identity permutation maximizes the expected hitting time of $n$, when the walk starts at 0. More generally, we prove that the hitting time of a random walk on a strongly connected $d$-directed graph is maximized when the graph is the line $[0,n]\cap\Z$ with $d-2$ self-loops at every vertex and $d-1$ self-loops at $0$ and $n$.
\newline
MSC 2010 \emph{subject classifications.}
Primary   60J10.   
\end{abstract}

\section{Introduction}\label{sec:intro}

Let $\sigma$ be a permutation of $\{0,\ldots,n\}$ and $(\xi_i)_i$ be i.i.d.\ uniform random variables in $\{-1,1\}$. We define the process $X^\sigma$ by setting $X^\sigma_0=0$ and $X^\sigma_{t+1} = \sigma(X^\sigma_t+\xi_{t+1})$ if $X^\sigma_t\neq 0,n$. Otherwise, if $X^\sigma_t=0$, then $X^\sigma_{t+1}$ is uniformly random in the set $\{\sigma(0), \sigma(1)\}$ and if $X^\sigma_t=n$, then it is uniformly random in the set $\{\sigma(n),\sigma(n-1)\}$ .

In this paper we address the question of maximizing the hitting time of~$n$ starting from~$0$ by the process $X^\sigma$. In particular we show that the identity permutation gives the slowest hitting time of~$n$ starting from~$0$, i.e.\ in this case $X^\sigma$ is a simple random walk on $\{0,\ldots,n\}$ with a self-loop at~$0$ and at~$n$.

\begin{theorem}\label{thm:perm-0n}
Let $\sigma$ be a permutation of $\{0,\ldots,n\}$ and $X^\sigma$ the Markov chain defined above. If $\tau_n = \inf\{t\geq 0: X^\sigma_t=n\}$, then
\[
\estart{\tau_n}{0} \leq n^2+n.
\]
Equality is achieved if and only if~$\sigma$ is the identity permutation.
\end{theorem}

As we explain in Section~\ref{sec:permutation}, the process $X^\sigma$ can be viewed as a random walk on a strongly connected graph where every vertex has outdegree and indegree equal to $2$. In Section~\ref{sec:directed} we prove a more general result (Theorem~\ref{thm:general}) concerning directed graphs in which every vertex has indegree equal to the outdegree equal to $d$. Then in Section~\ref{sec:permutation} we give the proof of Theorem~\ref{thm:perm-0n} by applying Theorem~\ref{thm:general} for $d=2$.

\begin{definition}\rm{
Let $d\geq 2$ and $n\in \N$. We define $L(d,n)$ to be the graph on~$[0,n]\cap \Z$ with the following properties:
\newline
1) $0$ and $n$ have $d-1$ self loops each and all other vertices have $d-2$ self loops
\newline
2) for every $k\neq 0$ and $\ell\neq n$ there is a directed edge from $k$ to $k-1$ and from $\ell$ to $\ell+1$.
}
\end{definition}

\begin{figure}[h!]
\begin{center}
\includegraphics[scale=1]{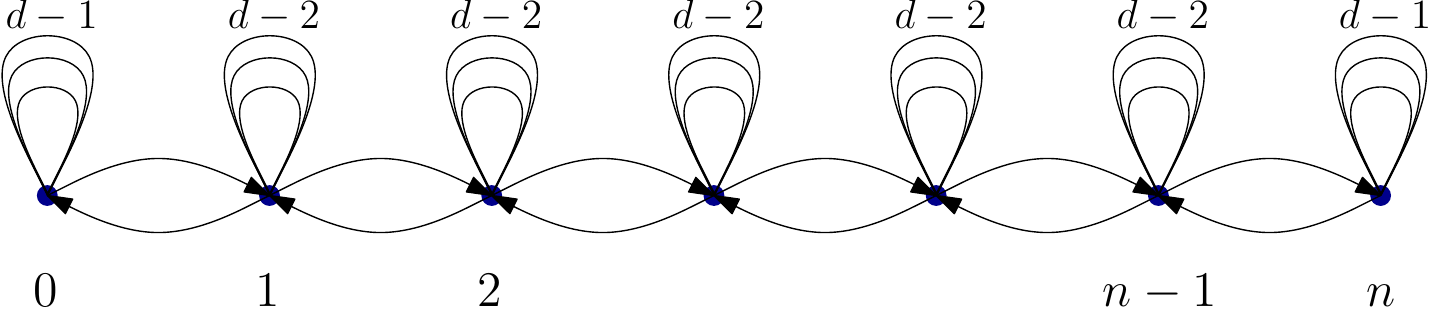}
\caption{The graph $L(d,n)$}\label{fig:ldn}
\end{center}
\end{figure}

\begin{definition}\rm{
A directed graph $G$ is {\em{connected}} if the graph~$G'$ obtained by removing the directions from the edges is connected. A directed graph $G$ is {\em{strongly connected}} if for every pair of vertices~$x,y$ there is a directed path from $x$ to $y$. We denote this by $x\rightsquigarrow y$.
We denote by $V(G)$ the vertex set of a graph~$G$ and
we write $\outda{x}{G}$ and $\ida{x}{G}$ for the outdegree and indegree of the vertex~$x$ in the graph $G$.
}
\end{definition}

\begin{theorem}\label{thm:general}
Let $G$ be a directed graph on $n+1$ vertices which is strongly connected and such that the indegree of every vertex is equal to its outdegree and equal to $d\ge2$ (allowing (multiple) self-loops and multiple edges). Then if $\tau_x$ is the first hitting time of $x$ by a simple random walk on $G$, then for all $x$ and $y$ we have
\[
\estart{\tau_y}{x} \leq \frac{d}{2}n(n+1).
\]
Equality is achieved if and only if $G$ is isomorphic to $L(d,n)$.
\end{theorem}

To date, much of the work on Markov chains has focused on random walks on undirected graphs.  Random walks on directed graphs have received relatively less attention and there are many interesting unexplored questions in this area. In particular, the first known bounds for mixing time parameters of a simple random walk on a directed graph have been studied by Fill in~\cite{JimFill} and for the Eulerian case by Montenegro in~\cite{Montenegro}.

Although the methods and ideas of the proofs are completely different, at a philosophical level the problem of maximizing the hitting time of $\{0,n\}$
by $X^\sigma$ over all permutations $\sigma$ is related to applications of rearrangement inequalities as in~\cite{BurSch} and~\cite{Wiener}. We state a related result that was proved by Aizenman and Simon in~\cite{AizenmanSimon}: among all open sets of equal area, the ball maximizes the exit time by a Brownian motion. An analogous statement for a discrete lazy random walk is proved in~\cite{SW12}.

\section{Directed graphs}\label{sec:directed}

In this section we give the proof of Theorem~\ref{thm:general}.
We start by stating a standard result about Eulerian graphs whose proof can be found in the discussion following Theorem~12 in~\cite{Bollobas}
. We then state and prove some preliminary results about directed graphs that will be used in the proof.

\begin{lemma}\label{lem:eulerian}
Let $G$ be a directed graph with the property that the outdegree of every vertex equals its indegree. If~$G$ is connected, then it is strongly connected.
\end{lemma}

\begin{lemma}\label{lem:norrisud}
Let $G$ be a finite strongly connected graph. Suppose there exists a vertex $i$ such that $\ida{i}{G}\geq \outda{i}{G}$ and for all $x\neq i$ we have $\ida{x}{G} \leq \outda{x}{G}$. If $\tau_i^+$ is the first return time to $i$ by a simple random walk on $G$, then
\[
\estart{\tau_i^+}{i} \leq \frac{\sum_x \outda{x}{G} + \outda{i}{G} - \ida{i}{G}}{\outda{i}{G}}.
\]
\end{lemma}

\begin{proof}[{\bf Proof}]

Since in every directed graph the sum of the outdegrees must equal the sum of the indegrees, we get
\begin{align}\label{eq:indoutd}
\ida{i}{G} -  \outda{i}{G} =\sum_{x\neq i} (\outda{x}{G} - \ida{x}{G}).
\end{align}
Consider the set $A=\{x: \ida{x}{G} <\outda{x}{G}\}$. We start adding fictitious edges from $i$ to all vertices $j \in A$ until the total number of edges that come into $j$ equals $\outda{j}{G}$. We call the new graph~$G'$ as shown in Figure~\ref{fig:fictitious}.
\begin{figure}[h!]
\begin{center}
\includegraphics[scale=1]{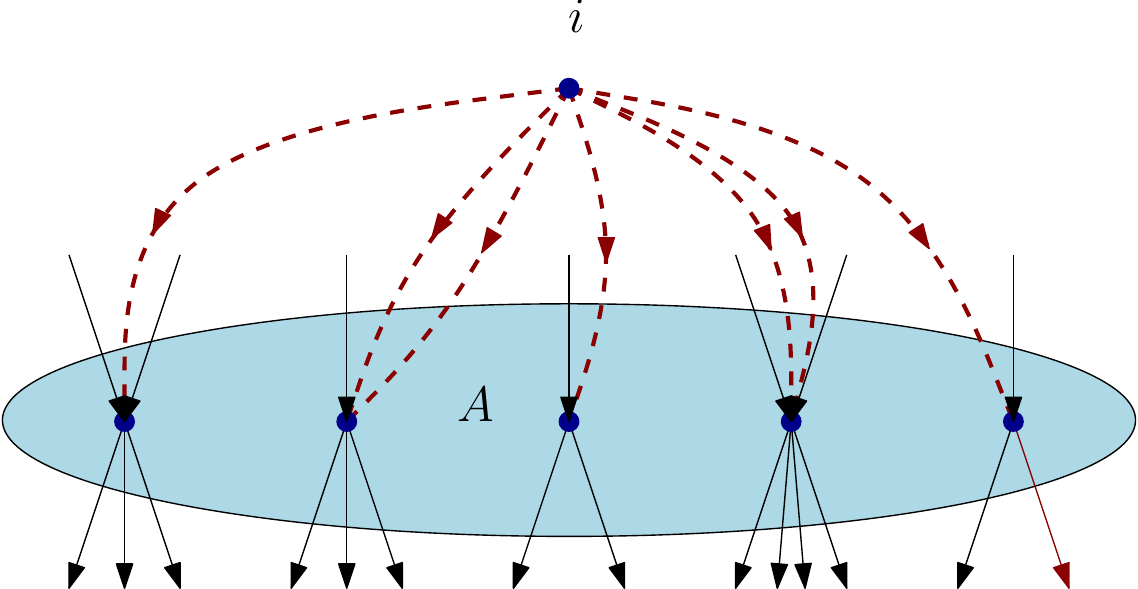}
\caption{The graph $G'$, where the red dashed lines are the new fictitious edges and the black lines are edges in the original graph.}\label{fig:fictitious}
\end{center}
\end{figure}

In view of~\eqref{eq:indoutd} in the new graph $\ida{i}{G}=\ida{i}{G'} = \outda{i}{G'}$. Also $\ida{x}{G'} = \outda{x}{G'}=\outda{x}{G}$ for all $x\in V(G)\setminus\{i\}$, and hence if $\pi$ denotes the stationary distribution of a simple random walk on the directed graph $G'$, we obtain
\[
\pi(i) = \frac{\ida{i}{G}}{\sum_{x\neq i}\outda{x}{G}+\ida{i}{G}}.
\]
Since the directed graph $G'$ is strongly connected, the simple random walk on $G'$ is irreducible, and hence the return time to $i$ in the graph $G'$ is $1/\pi(i)$, i.e.,
\begin{align*}
\frac{\sum_{x\neq i}\outda{x}{G} + \ida{i}{G}}{\ida{i}{G}} = \frac{\outda{i}{G}}{\outda{i}{G'}} \estart{\tau_i^+}{i}  + \sum_{j\in A} \frac{(\ida{j}{G'}-\ida{j}{G})s}{\outda{i}{G'}}(1+\estart{\tau_i}{j}),
\end{align*}
where $\tau_i$ is the hitting time of vertex $i$ in $G'$.
However, since $\outda{i}{G'}=\ida{i}{G}$ and $\estart{\tau_i}{j}\geq 1$ for all $j\neq i$ by rearranging we get
\[
\estart{\tau_i^+}{i} \leq \frac{\sum_x \outda{x}{G} + \outda{i}{G} - \ida{i}{G}}{\outda{i}{G}}
\]
and this concludes the proof of the lemma.
\end{proof}

In the next results, we will usually need to construct new graphs that come from a directed graph $G$ with a distinguished vertex $u$. In order to avoid repetitions of the same construction in many of the statements and proofs we now give the definition of the new graph.

\begin{definition}\label{def:graphgi}\rm{
Let $G$ be a directed graph and $u\in V(G)$ a distinguished vertex. We write~$I_u$ for the set of vertices having a directed edge to~$u$. For each $i\in I_u$, we construct the graph $G_i$ as follows: first we remove $u$ from the graph $G$ together with all the edges incident to it and then we connect every $j\in I_u\setminus\{i\}$ to
$i$ using multiple edges if there are multiple edges in the original graph between~$j$ and~$u$ so that $\outda{j}{G_i}=\outda{j}{G}$. We define $A_i =  \{x: i\rightsquigarrow x \ \text{ in } \ G_i\}$ and the graph $(A_i,E_i)$ to be the subset of~$G_i$ induced by $A_i$. We write $\ida{x}{A_i}$ and $\outda{x}{A_i}$
for the indegree and outdegree of $x \in A_i$ in the graph $(A_i,E_i)$.
}
\end{definition}

\begin{figure}[h!]
\begin{center}
\includegraphics[scale=1]{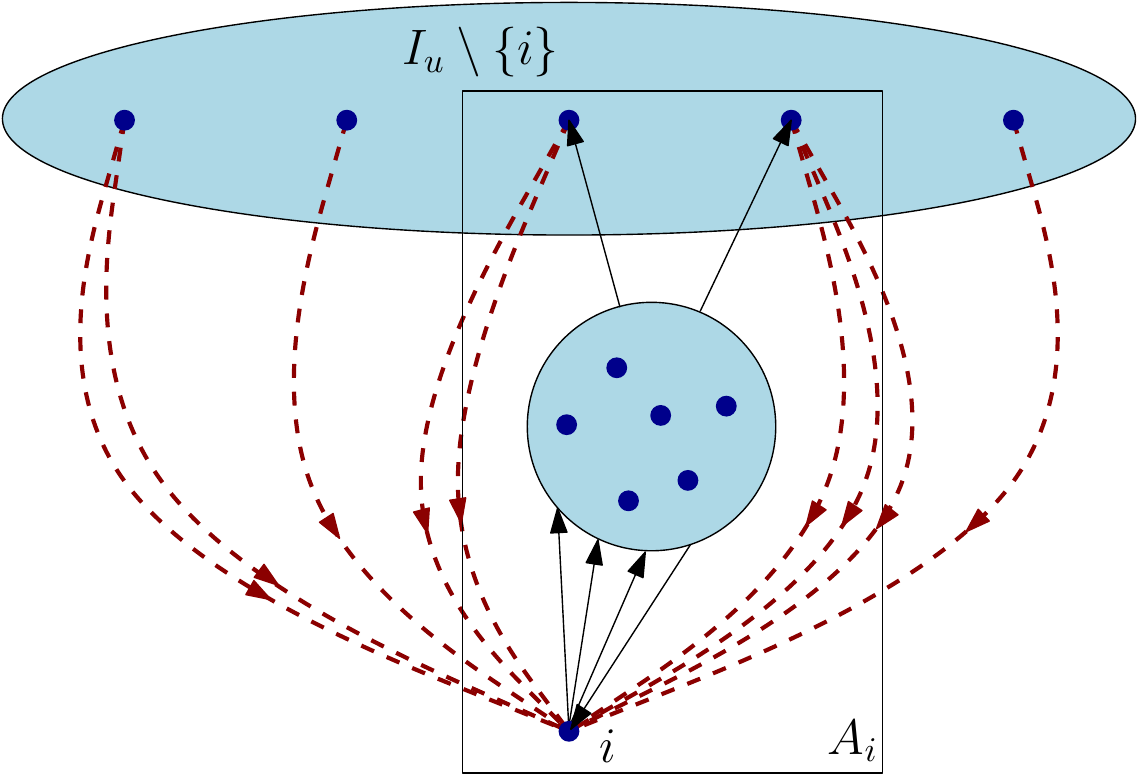}
\caption{The graph $(A_i,E_i)$, where the red dashed lines represent the new edges and the black lines are edges in the original graph.}\label{fig:defnai}
\end{center}
\end{figure}

\begin{lemma}\label{lem:setA}
Let $G$ be a strongly connected graph such that~$\ida{x}{G}=\outda{x}{G}$ for all~$x$. Let $u$ be a distinguished vertex of $G$.
Fix $i$ and suppose that $A_i\neq \emptyset$.
The graph~$(A_i,E_i)$ is strongly connected and contains $i$. Furthermore if $r_i$ is the number of directed edges from $i$ to $u$ in the graph $G$, then
\begin{align*}
\outda{x}{A_i} &=\outda{x}{G} \ \text{ and } \  \ida{x}{A_i} \leq \ida{x}{G}, \ \text{ for all } x\in A_i\setminus\{i\} \\
\outda{i}{A_i}&= \outda{i}{G}-r_i \ \text{ and } \ \ida{i}{A_i}\geq \outda{i}{A_i}.
\end{align*}
\end{lemma}

\begin{proof}[{\bf Proof}]

First we establish that if $x\in A_i$, then there is a directed path from $x$ to $i$ using only vertices of $A_i$. Indeed, in the original graph $G$, there is a path from $x$ to $i$, since $G$ was assumed to be strongly connected. If this path does not use the vertex $u$, then we have nothing to show. If it does, then if it uses the edge $(i,u)$, then we are done again. If not, then it uses an edge of the form $(i_\ell,u)$, in which case since $i_\ell \in I_{u}$ is connected to $i$ by at least one edge in $G_i$, it follows that $x\to i$.
Clearly by the definition of the set $A_i$ all the vertices in the path from $x$ to $i$ are in $A_i$. Furthermore, $i \in A_i$, since its neighbours are in $A_i$ by definition. Since all vertices in $A_i$ are connected to $i$ in both directions, it follows that the graph defined by $A_i$ is strongly connected.

Again by definition it follows that all the neighbours of $x\in A_i\setminus\{i\}$ are in $A_i$. Hence if~$x\in A_i\setminus\{i\}$ we have $\outda{x}{A_i} =\outda{x}{G}$  and $\outda{i}{A_i}=\outda{i}{G}-r_i$. th we deduce that $\ida{x}{A_i}\leq \ida{x}{G}$ for all~$x\in A_i\setminus\{i\}$. Using these inequalities together with the fact that
\[
\sum_{x\in A_i\setminus\{i\}} (\outda{x}{A_i} - \ida{x}{A_i}) =\ida{i}{A_i} - \outda{i}{A_i}
\]
we deduce that
$\ida{i}{A_i} \geq \outda{i}{A_i}$ and this finishes the proof of the lemma.
\end{proof}

\begin{lemma}\label{lem:excursions}
Let $G$ be a strongly connected graph on $n$ vertices such that for all vertices~$x$ it satisfies $\ida{x}{G} = \outda{x}{G} =d$. We fix a vertex $u$ and write $I_u$ for the set of vertices that have an edge leading to $u$. If $\mu$ is a probability measure supported on $I_u$ and $\tau_u$ is the first hitting time of $u$ by a simple random walk on $G$, then
\[
\estart{\tau_{u}}{\mu} \leq nd-d.
\]
\end{lemma}

\begin{proof}[{\bf Proof}]

Let $X$ be a simple random walk on $G$ with $X_0\sim \mu$ and $I_u=\{i_1,\ldots,i_k\}$ with $k\leq d$, since there could be multiple edges.

For any $i\in I_u$ we write $r_{i}$ for the number of directed edges from $i$ to $u$.
Every time the random walk is at a vertex $i\in I_u$ it has probability $r_i/d$ of jumping directly to $u$. If it does not jump, then it starts walking in the remaining graph until the first time that it hits $I_u$ again. Define $(\xi_{k}^{(i)})_k$ to be the lengths of i.i.d.\ ``excursions'' starting from $i\in I_u$ until the first time that they come back to the set $I_u$ without hitting $u$ independently for different~$i$.

It is clear that adding directed edges from every $\ell\in I_u\setminus\{i\}$ to $i$ cannot affect $\xi_{1}^{(i)}$. Hence we can upper bound $\xi_{1}^{(i)}$ by the return time to $i$ in the graph $(A_i,E_i)$ constructed in Definition~\ref{def:graphgi}. In this graph we have $\outda{i}{A_i} = d-r_i$.

Lemma~\ref{lem:setA} gives that $A_i$ satisfies the assumptions of Lemma~\ref{lem:norrisud}. Therefore since $|A_i|< n$ and $r_i\geq 1$ for all $i \in I_u$ we deduce
\begin{align}\label{eq:upperboundi}
\estart{\xi^{(i)}_1}{i} \leq \frac{dn-d-1}{d-r_i}.
\end{align}
We now define independent collections of random variables
\begin{align*}
B_1^{(i_1)},B_2^{(i_1)},&\ldots \ \ \text{ i.i.d.\ $\B(r_{i_1}/d)$} \\
B_1^{(i_2)},B_2^{(i_2)},&\ldots \ \ \text{ i.i.d.\ $\B(r_{i_2}/d)$} \\
&\vdots \\
B_1^{(i_k)},B_2^{(i_k)},&\ldots \ \ \text{ i.i.d.\ $\B(r_{i_k}/d)$},
\end{align*}
where $\B(p)$ stands for the Bernoulli distribution of parameter $p$.

We can realize the random walk $X$ until the first time that it hits $u$ in the following way: at time~$0$ if $B_1^{(X_0)}=1$, then it jumps directly to~$u$. Otherwise it makes an ``excursion'' of length $\xi_1^{(X_0)}$ until the first time that it comes back to $I_u$.
We define $\ell(1)=X_0$ and $\zeta_1 = \xi_1^{(X_0)}$. Inductively we define
\[
\ell(k+1) =  X_{\zeta_k} \ \text{ and } \ \zeta_{k+1} = \sum_{s=1}^{k+1} \xi_s^{(\ell(s))}.
\]
In words, $\zeta_k$ is the time which has passed until the end of the $k$-th excursion and $\ell(k)$ is the position of the random walk at the end of the $(k-1)$-th excursion.
At the end of the $(k-1)$-th ``excursion''~$X$ hits~$u$ directly with probability $r_{\ell(k)}/d$. If it does not, thene we attach another excursion of length $\xi^{(\ell(k))}_k$ and we continue in the same way until the first time that $X$ hits $u$.
 We finally define
\[
T=\min\{ s: B_s^{(\ell(s))} =1\},
\]
i.e.\ $T$ is the number of used Bernoulli random variables until the first time that a Bernoulli is equal to~$1$. Hence we can now write
\begin{align}\label{eq:decod}
\estart{\tau_{u}}{\mu} = \E{\zeta_{T-1}} + 1.
\end{align}
By the definition of $\zeta$ we get
\begin{align}\label{eq:bigeqn}
\E{\zeta_{T-1}} &= \E{\sum_{k=1}^{T-1} \xi_k^{(\ell(k))}} =  \sum_{k=1}^{\infty}\E{\xi_k^{(\ell(k))} \1(T>k)} =\sum_{k=1}^{\infty} \E{\econd{\xi_k^{(\ell(k))} \1(T>k)}{\xi_k^{(\ell(k))}, (\ell(j))_{j\leq k}}} \\ \nonumber
&= \sum_{k=1}^{\infty} \E{\frac{(d-r_{\ell(1)})}{d}\ldots \frac{(d-r_{\ell(k)})}{d}  \econd{\xi_k^{(\ell(k))}}{(\ell(j))_{j\leq k}}}.
\end{align}
Using~\eqref{eq:upperboundi} we now immediately get that
\[
\econd{\xi_k^{(\ell(k))}}{(\ell(j))_{j\leq k}} \leq \frac{dn-d-1}{d-r_{\ell(k)}},
\]
since $\xi_k^{(\ell(k))}$ is the length of an ``excursion'' started from the vertex $\ell(k)$. Hence plugging that into~\eqref{eq:bigeqn} and using that $r_i\geq 1$ for all $i$ we get
\begin{align*}
\E{\zeta_{T-1}} \leq \sum_{k=1}^{\infty} (dn-d-1)\frac{(d-1)^{k-1}}{d^k} = dn-d-1.
\end{align*}
This together with~\eqref{eq:decod} gives
\[
\estart{\tau_{u}}{\mu} \leq dn-d
\]
and this concludes the proof of the lemma.
\end{proof}

We are now ready to give the proof of Theorem~\ref{thm:general}.

\begin{proof}[{\bf Proof of Theorem~\ref{thm:general}}]

Note that if $G$ is isomorphic to $L(d,n)$, then
\[
\max_{x,y}\estart{\tau_y}{x} = \frac{d}{2}n(n+1).
\]
We prove the strict inequality by induction on $n$. For $n=1$ it is trivially true. Suppose that for any strongly connected graph on $n+1$ vertices with $\id{x}=\outd{x}=d$ for all $x$ and which is not isomorphic to $L(d,n)$ we have
\[
\max_{x,y}\estart{\tau_y}{x} < \frac{d}{2}n(n+1).
\]
Let $G'$ be a strongly connected graph on $n+2$ vertices with~$\ida{x}{G'}=\outda{x}{G'}=d$ for all~$x$ which is not isomorphic to $L(d,n+1)$. We will show that for all $x$ and $y$
\begin{align}\label{eq:xyhit}
\estart{\tau_{y}}{x} < \frac{d}{2}(n+1)(n+2).
\end{align}
Let $I=\{i_1,\ldots,i_k\}$ be the vertices such that there is a directed edge from every $i_\ell$ to $y$. Note that $k\leq d$, since we are allowing multiple edges and self-loops.

Clearly we can write
\begin{align}\label{eq:firsteq}
\estart{\tau_{y}}{x} = \estart{\tau_{I}}{x} + \estart{\tau_{y}}{\mu},
\end{align}
where $\tau_I$ is the first hitting time of the set~$I$ by the simple random walk on $G'$ and~$\mu$ is a measure on~$I$ with
\[
\mu(i_\ell) = \prstart{X_{\tau_{I}} =i_{\ell}}{x}.
\]
By Lemma~\ref{lem:excursions} we immediately obtain
\begin{align}\label{eq:excproof}
\estart{\tau_{y}}{\mu} \leq nd+d.
\end{align}
If $x\in I$, then~\eqref{eq:excproof} finishes the proof.
So from now on we assume that $x\notin I$.
 Since in order to hit~$y$ we must first hit the set $I$, we are going to look at the graph not containing the vertex $y$ and the edges incident to it. Clearly adding edges coming out of points of $I$ is not going to change the first hitting time of the set $I$ starting from $x$.
We now explain how we add extra edges coming out of the set $I$ in order to apply the induction hypothesis to a graph of smaller size.

Let $J=\{j_1,\ldots,j_m\}$ be the vertices such that there is a directed edge from~$y$ to every~$j_\ell$. Note that as above $m\leq d$. Removing the vertex~$y$ and its incident edges removes the edges from $y$ to vertices in $J$ as well as edges from vertices in $I$ to $y$. Therefore in order to keep the in and out degrees of all the vertices in the new graph (obtained by removing~$y$) equal to $d$ we shall add edges going from~$I$ to~$J$. We describe how to achieve this whilst keeping the graph connected.

Let $J_1$ be the subset of $J$ containing only those vertices $j_r$ which in graph~$G'$ have an undirected path from~$j_r$ to~$i_1$ that does not visit vertex~$y$.

Suppose we have defined the sets $J_1,\ldots, J_\ell$. We next define $J_{\ell+1}$ to be the set of $j_r \in J \setminus \cup_{s\leq \ell}J_s$ which in $G'$ have an undirected path from $j_r$ to $i_{\ell+1}$ that does not visit vertex~$y$.

Note that some of the sets $J_\ell$ could be empty, nevertheless since the original graph is connected we have
\[
\bigcup_r J_r = J.
\]
We now consider only the non-empty subsets $J_r$, which we index by $s_1,\ldots,s_\ell$ so that if $j\in J_{s_1}$, then it is connected to $i_{s_1}$. We then connect $i_{s_1}$ to an element of $J_{s_2}$ (chosen arbitrarily) and $i_{s_2}$ to an element of $J_{s_3}$ and so on. Finally we connect $i_{s_m}$ to an element of $J_{s_1}$.

At the end of this procedure, we add edges between~$I$ and~$J$ so that in the resulting graph every vertex has indegree equal to outdegree equal to~$d$. This is possible, since the indegree of~$y$ is equal to its outdegree. We call the resulting graph~$G''$ as shown in Figure~\ref{fig:g''}.

\begin{figure}[h!]
\begin{center}
\includegraphics[scale=1]{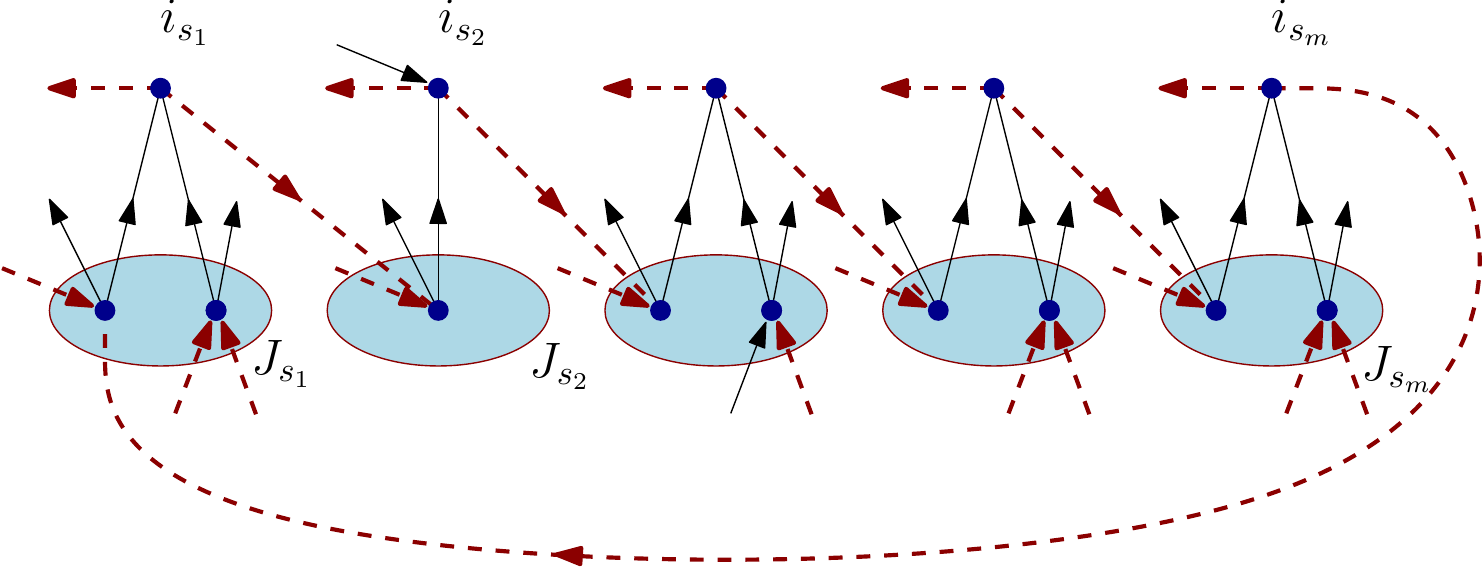}
\caption{The graph $G''$, where the red dashed lines are the new edges and the black lines are edges in the original graph $G'$.}\label{fig:g''}
\end{center}
\end{figure}

We now claim that this new graph~$G''$ is connected. Indeed, all the vertices in $J_r$ are connected to each other by the definition of the set $J_r$. Let $j_1\in J_{s_k}$ and $j_2 \in J_{s_{k+1}}$. Then since we connect $j_1$ to $i_{s_{k+1}}$ it follows that $j_1$ is connected to $j_2$. Furthermore, as each~$i\in I$ is connected to at least one~$j\in J$, the graph is connected. By Lemma~\ref{lem:eulerian} it follows that it is strongly connected (since we kept the in and out degrees at every vertex equal to~$d$) on~$n+1$ vertices.

If the graph~$G''$ is not isomorphic to~$L(d,n)$, then by the induction hypothesis we get
\[
\estart{\tau_I}{x} <  \frac{d}{2}n(n+1)
\]
and this together with~\eqref{eq:firsteq} and~\eqref{eq:excproof} finishes the proof of~\eqref{eq:xyhit} in this case. If~$G''$ is isomorphic to~$L(d,n)$ (in which case we identify these two graphs), then we shall consider two separate cases:
$|I|=1$ and~$|I| \geq 2$. We start with the case~$|I|=1$. If~$I=\{i_1\}$ and from~$i_1$ the only vertex we can reach in one step is $y$ in~$G'$, then
\begin{align}\label{eq:easyd}
\estart{\tau_y}{i_1} \leq d.
\end{align}
Since the graph~$G''$ is a strongly connected graph on $n+1$ vertices with in and out degree of every vertex equal to $d$, from the induction hypothesis it follows that
\[
\estart{\tau_I}{x} \leq \frac{d}{2}n(n+1).
\]
Hence this together with~\eqref{eq:firsteq} and~\eqref{eq:easyd} gives that in this case
\[
\estart{\tau_y}{x} <\frac{d}{2}(n+1)(n+2).
\]
If $i_1$ has another out-neighbour~$h\neq y$ in~$G'$, then $i_1$ cannot be an endpoint of the line, so the hitting time of~$i_1$ will be bounded by the maximum hitting time on $L(d,m)$ for $m\leq n-1$ and thus we get
\[
\estart{\tau_{i_1}}{x} < \frac{d}{2}n(n+1).
\]
This finishes the proof of~\eqref{eq:xyhit} in the case~$|I|=1$.
It remains to show that if~$|I|\geq 2$, then
\[
\estart{\tau_I}{x} < \frac{d}{2}n(n+1).
\]
Since the subgraph~$G''$ is isomorphic to~$L(d,n)$, then the vertex $i^*$ of $I$ closest to~$x$ satisfies
\[
\estart{\tau_{i^*}}{x}<\frac{d}{2}n(n+1).
\]
This together with~\eqref{eq:excproof} finishes the proof of the theorem.
\end{proof}

\section{Permutation walk}\label{sec:permutation}

\begin{lemma}\label{lem:connperm}
Let $\sigma$ be a permutation of $\{0,\ldots,n\}$. Then the Markov chain $X^\sigma$ is irreducible.
\end{lemma}

\begin{proof}[{\bf Proof}]
We first observe that the Markov chain can be represented as a random walk on a directed graph such that the outdegree of every vertex is equal to its indegree (note that we also count self-loops). Hence if we establish that ignoring orientations, the underlying graph is connected, then we can apply Lemma~\ref{lem:eulerian} and finish the proof.

It is easy to see that the undirected graph is connected. Indeed, from the description of the process, all the odd points are connected to each other and all the even points are connected to each other, since $k-1$ and $k+1$ both lead to $\sigma(k)$. Since $0$ and $1$ both lead to $\sigma(0)$, it follows that the two sets (odd and even points) are connected, and hence this concludes the proof.
\end{proof}

\begin{proof}[{\bf Proof of Theorem~\ref{thm:perm-0n}}]

As we already noted in the proof of Lemma~\ref{lem:connperm} above, the Markov chain~$X^\sigma$ can be viewed as a random walk on a directed graph such that the in and out degree of every vertex is equal to $2$. Furthermore, from Lemma~\ref{lem:connperm} we know that this graph is strongly connected.

Hence applying Theorem~\ref{thm:general} shows that
\[
\estart{\tau_n}{0} \leq n^2+n.
\]
From Theorem~\ref{thm:general} we get that equality is achieved only if the resulting graph is isomorphic to~$L(2,n)$ and $\sigma(0)=0$ and $\sigma(n)=n$. It thus follows that~$\sigma$ has to be the identity permutation and this finishes the proof.
\end{proof}

\begin{remark}\rm{
We note that the statement of Theorem~\ref{thm:perm-0n} remains true if we change the Markov chain as follows: whenever at $k$ the next step is either $\sigma(k)+1$ or $\sigma(k)-1$ equally likely. Indeed, it is easy to see that  this Markov chain is again a simple random walk on a $2$ directed graph which is strongly connected, and hence Theorem~\ref{thm:general} applies.
}
\end{remark}

\section{Open problem}\label{sec:motivation}

The following problem was communicated to us by Yuval Peres, but we could not trace its origins. We state it here:

\textbf{Open problem:} Let $\sigma$ be a permutation of $\{-n,\ldots,n\}$. Let $(\xi_i)_i$ be i.i.d.\ taking values in $\{-1,1\}$ equally likely and set $X^\sigma_0=0$ and $X_{t+1}^\sigma = \sigma(X_t^\sigma + \xi_{t+1})$ if $X_t^\sigma \neq n, -n$, otherwise $X_{t+1}^\sigma$ takes values in $\{\sigma(n), \sigma(n-1)\}$ or $\{\sigma(-n), \sigma(-n+1)\}$ respectively equally likely. Show that the identity permutation maximizes $\estart{\tau_{\{-n,n\}}}{0}$, where $\tau_{\{-n,n\}}$ is the first hitting time of the set $\{-n,n\}$ by~$X^\sigma$.

\begin{remark}\rm{
In contrast to Theorem~\ref{thm:perm-0n} where equality is achieved only when~$\sigma$ is the identity permutation, we note that for this problem this is no longer the case. In other words, the identity is not the unique permutation that maximizes $\estart{\tau_{\{-n,n\}}}{0}$. Indeed, if $\sigma(x)=-x$ for all $x\in [-n,n]\cap \Z$, then $\estart{\tau_{\{-n,n\}}}{0} = n^2$. Also, it is easy to check that if~$\sigma$ is the permutation that transposes~$0$ and~$1$, then it also achieves the same upper bound. Nevertheless, if $\sigma$ only transposes~$k$ with~$k+1$, then
\[
\estart{\tau_{\{-n,n\}}}{0} = n\frac{2n-3-2k}{n-1} + n(n-2)
\]
for all~$0<k\neq n-1,n$.
}
\end{remark}

By arguing in a similar way as in the proof of Lemma~\ref{lem:connperm}, it is easy to see that the process $X^\sigma$ can be viewed as a random walk on a $2$ directed graph which is strongly connected. Hence by Theorem~\ref{thm:general} we immediately get  that for any permutation $\sigma$
\[
\estart{\tau_{\{-n,n\}}}{0} \leq 4n^2 + 6n + 2.
\]

\section*{Acknowledgements}
We are grateful to Yuval Peres for telling us about the open problem stated in Section~\ref{sec:motivation}. We thank Jason Miller and James Norris for helpful discussions. Part of this work was completed while the first author was a postdoc at the University of Angers, France.

\bibliographystyle{plain}
\bibliography{biblio}

\end{document}